\documentclass{amsart}
\usepackage{amsmath,amssymb}
\usepackage{amsthm}
\usepackage{bm}
\usepackage{graphicx}
\usepackage{epstopdf}
\usepackage{color}
\newtheorem{definition}{Definition}[section]
  \newtheorem{theorem}[definition]{Theorem}
   \newtheorem{lemma}[definition]{Lemma}
  \newtheorem{corollary}[definition]{Corollary}
  \newtheorem{proposition}[definition]{Proposition}

\newtheorem{question}[definition]{Question}

\newcommand{\fix}{{\rm Fix}}

\newcommand{\Z}{{\mathbb Z}}

\begin{document}
\title{Meridional rank and bridge number for a class of links} 

\author{Michel Boileau}\thanks{First author partially supported by ANR
projects 12-BS01-0003-01 and 12-BS01-0004-01}
\address{Aix-Marseille Universit\'e, CNRS, Centrale Marseille, I2M, UMR 7373,
13453 Marseille, France}
\email{michel.boileau@univ-amu.fr}

\author{Yeonhee Jang}
\address{Department of Mathematics, Nara Women's University, Nara, 630-8506 Japan}
\email{yeonheejang@cc.nara-wu.ac.jp}

\author{Richard Weidmann}
\address{Mathematisches Seminar, Christian-Albrechts-Universität zu Kiel, Ludewig-Meyn Str.~4, 24098 Kiel,
Germany}
\email{weidmann@math.uni-kiel.de}

\begin{abstract}
We prove that links with meridional rank 3 whose 2-fold branched covers are graph manifolds are 3-bridge links.
This gives a partial answer to a question by S. Cappell and J. Shaneson
on the relation between the bridge numbers and meridional ranks of links.
To prove this, we also show that the meridional rank of any satellite knot is at least $4$.
\end{abstract}

\maketitle

%%%%%%%%%%%%%%%%%%%%%%%%%%%%%%%%%%%%%%%%%%%%%%%%%%%%%%%%

\section{Introduction}

An {\it $n$-bridge sphere} of a link $L$ in the 3-sphere $S^3$
is a 2-sphere which meets $L$ in $2n$ points
and cuts $(S^3, L)$ into $n$-string trivial tangles. 
Here, an {\it $n$-string trivial tangle} is 
a pair $(B^3, t)$ of the $3$-ball $B^3$ and 
$n$ arcs properly embedded in $B^3$ parallel to the boundary of $B^3$. 
It is known that every link admits an $n$-bridge sphere for some positive integer $n$.
We call a link $L$ an {\it $n$-bridge link} 
if $L$ admits an $n$-bridge sphere and does not admit an ($n-1$)-bridge sphere.
We call $n$ the {\it bridge number} of the link $L$
and denote it by $b(L)$.

If a link admits an $n$-bridge sphere, 
then it is easy to see that $\pi_1(S^3\setminus L)$
can be generated by $n$ meridians, where a meridian is an element of the fundamental group that is represented by a curve that is freely homotopic to a meridian of $L$.
This implies that the minimal number {of meridians needed to generate the group $\pi_1(S^3\setminus L)$
is less than or equal to $b(L)$.
We denote by $w(L)$ the minimal number of meridians of $\pi_1(S^3\setminus L)$ and call it the {\it meridional rank} of $L$. Thus for any link $L$ we have $b(L)\ge w(L)$.

S. Cappell and J. Shaneson \cite[pb 1.11]{Kir}, as well as K. Murasugi, 
have asked whether the converse holds:

\begin{question}\label{question-wirtinger}
Does  the equality $b(L)=w(L)$ hold for any link $L$ in $S^3$?
\end{question}

This is known to be true for (generalized) Montesinos links by \cite{Boi3}, torus links by  \cite{Ros} and for another class of knots (also refered to as generalized Montesinos knots) by \cite{LM}. More recently the equality has been established for a large class of iterated torus knots using knot contact homology \cite{CH}, see also \cite{Co}.
It is a consequence of Dehn's Lemma that $b(L)=1$ if and only if $w(L)=1$. Moreover
in \cite{Boi6} it is proved that $b(L)=2$ if and only if $w(L)=2$.
%We recall their results in Section \ref{sec-pre}.
%We recall the definition of arborescent links in Section \ref{sec-arb} 
%and give an outline of the proof for the following theorem in Section \ref{sec-main}.

The main purpose of this paper is to prove the following theorem.

\begin{theorem}\label{thm-main}
Let $L$ be a link in the $3$-sphere $S^3$, and suppose that the 2-fold branched cover of $S^3$ branched along $L$ is a graph manifold.
If $w(L)=3$, then $b(L)=3$, i.e., $L$ is a $3$-bridge link.
\end{theorem}

The above theorem, together with the result in \cite{Boi6},
implies that $b(L)=3$ if and only if $w(L)=3$ for links whose 2-fold branched covers are graph manifolds. In particular we obtain the following:
%Here, a graph manifold is a 3-manifold obtained by gluing some Seifert fibered spaces, i.e., some manifolds foliated by circles.

\begin{corollary}\label{cor-4bridges}
Let $L \subset S^3$ be a link whose 2-fold branched cover is a graph manifold. If $b(L) = 4$, then $w(L) = 4$.
\end{corollary}

We obtain also the following corollary which answers a question posed in \cite[Question 2]{Boi8} positively for graph manifolds.

\begin{corollary}\label{cor-inversion}
For a closed orientable graph manifold $M$, any inversion of $\pi_1(M)$ is hyper-elliptic.
\end{corollary}

We remark that Cappell and Shaneson's question (Question \ref{question-wirtinger})
is related, by taking the $2$-fold branched covering, to the question whether or not 
the Heegaard genus of a 3-manifold equals the rank of its fundamental group.
For the latter question many counter-examples are known, see
\cite{Boi9, Boi10, Boi11, Li,SW,W1}.
Thus there exist manifolds such that the ranks of their fundamental groups 
are smaller than their Heegaard genera.
To the question of Cappell and Shaneson, however, no counter-examples is known to date.

We also remark that if we replace $w(L)$ with the rank of the link group $\pi_1(S^3\setminus L)$
then we can easily find examples where the differences between the two numbers are arbitrarily large.
For example, the rank of the group $\pi_1(S^3\setminus K(p,q))$ of a torus link $K(p,q)$ is 2
while $b(K(p,q))={\rm min}(p,q)$ by \cite{Sch2}.

To prove Theorem \ref{thm-main} we distinguish two cases, namely the case when the link $L$ is arborescent in the sense of Bonahon and Siebenmann \cite{Bon} and the case when $L$ is not arborescent. We will make use of the following theorem, which is interesting in its own right.

\begin{theorem}\label{thm-3meridians} Let $K$ be a prime knot such that $S^3\backslash K$ has a nontrivial JSJ-decomposition and let $m_1,m_2,m_3$ be meridians. Then one of the following holds:
\begin{enumerate} 
\item $\langle m_1,m_2,m_3\rangle$ is free.
\item $\langle m_1,m_2,m_3\rangle$ is conjugate into the subgroup of $\pi_1(S^3\backslash K)$ corresponding to the peripheral piece of $S^3\backslash K$.
\end{enumerate}
\end{theorem}

\begin{corollary}\label{cor-satellite} Let $K$ be a prime knot such that $S^3\backslash K$ has a nontrivial JSJ-decomposition. Then
$w(K)\geq 4$.
\end{corollary}

\begin{corollary}\label{cor-3meridians} Let $K \subset S^3$ be a knot. If $w(K) \leq 3$, then $K$ is either a hyperbolic knot or a torus knot or a connected sum of two 2-bridge knots.
\end{corollary}

Theorem \ref{thm-3meridians} suggests the following strengthening of Question \ref{question-wirtinger} for a hyperbolic knot:

\begin{question}\label{question-free}
Let $K \subset S^3$ be a hyperbolic knot. Is a subgroup of  $\pi_1(S^3\setminus K)$ generated by at most $b(K) - 1$ meridians free?
\end{question}

In the case of torus knots the conclusion of Question~\ref{question-free}  has been established  by M. Rost and H. Zieschang (see \cite{Ros}).
The case of hyperbolic 3-bridge knots follows from a general result for subgroups generated by two meridians in a knot group, see Proposition \ref{propNE} in  Section \ref{sec-free}. It should be noted that the conclusion of Question~\ref{question-free} does obviously not hold for connected sums of knots, it is moreover not difficult to come up with examples of prime knots with nontrivial JSJ-decomposition for which the conclusion does not hold either.

\smallskip There is a natural partial order on the set of links in $S^3$ given by degree-one maps: We say that a link $L \subset S^3$ {\it dominates} a link $L' \subset S^3$ and write $L \geq L'$ if there is a proper degree-one map 
$f: E(L) \to E(L')$ between the exteriors of $L$ and $L'$ whose restriction to the boundary is a homeomorphism which extends to the regular neighborhoods of $L$ and $L'$. It defines a partial order on the set of links in $S^3$, and it is an open problem to characterize minimal elements. In particular  the behavior  of the bridge number with respect to this order is far from being understood:

\begin{question}\label{bridgeversusorder} Let $L$ and $L'$ be links in $S^3$. Does $L \geq L'$ imply $b(L) \geq b(L')$?
\end{question}

It follows from the definition that the epimorphism $f_{\star}: \pi_1(S^3\setminus L) \to \pi_1(S^3\setminus L')$ induced by the degree-one map $f: E(L) \to E(L')$ preserves the meridians and so that $w(L) \geq w(L')$ whenever $L \geq L'$. Therefore an affirmative  answer to Question \ref{question-wirtinger} would imply an affirmative answer to Question \ref{bridgeversusorder}.

The answer to Question \ref{bridgeversusorder} is certainly positive when $b(L') = 2$ as in this case any knot $L$ with $L \geq L'$ cannot be trivial. Our results moreover imply the following:

\begin{proposition}\label{prop-degreeone} Let $L \geq L'$ be two  links in $S^3$.

\noindent a) If $b(L') = 3$, then $b(L) \geq 3$.

\noindent b) If $b(L') = 4$ and the 2-fold cover of $S^3$ branched along $L$ is a graph manifold, then $b(L) \geq 4$.
\end{proposition}

In Section \ref{sec-prop1}, we recall the definition and some properties of arborescent links and show that an arborescent link $L$  with $w(L) = 3$ is hyperbolic. Section \ref{sec-thm} is devoted 
to the proof of Theorem \ref{thm-main} for arborescent links. Section \ref{sec-free} contains the proof of Theorem \ref{thm-3meridians}. In Section \ref{sec-final} we complete the proof of Theorem \ref{thm-main} for the case of non-arborescent links. Then Section \ref{sec-degreeone} contains the proof of Proposition \ref{prop-degreeone}.

%%%%%%%%%%%%%%%%%%%%%%%%%%%%%%%%%%%%%%%%%%%%%%%%%%%%%%%%

\section{Arborescent links}\label{sec-prop1}

A (3,1)-{\it manifold pair} 
is a pair $(M,L)$ of a compact oriented 3-manifold $M$
and a proper 1-submanifold $L$ of $M$.
By a {\it surface $F$ in $(M,L)$},
we mean a surface $F$ in $M$ intersecting $L$ transversely.
Two surfaces $F$ and $F'$ in $(M,L)$ are said to be {\it pairwise isotopic} ({\it isotopic}, in brief,)
if there is a homeomorphism $f:(M,L)\rightarrow(M,L)$ such that
$f(F)=F'$ and $f$ is pairwise isotopic to the identity.
We call a (3,1)-manifold pair a {\it tangle} if $M$ is homeomorphic to $B^3$. 
A {\it trivial tangle} 
is a (3,1)-manifold pair $(B^3, L)$,
where $L$ is the union of two properly embedded  arcs in the 3-ball $B^3$
which  together with arcs on the boundary of $B^3$ bound disjoint disks.
A {\it rational tangle} is a trivial tangle $(B^3, L)$ endowed with a homeomorphism from $\partial (B^3, L)$ to the \lq\lq standard\rq\rq\ pair of the 2-sphere and the union of four points on the sphere.
%hence it is
%a pair $(B^3, L)$,
%where $L$ is the union of two properly embedded unknotted  arcs in the 3-ball $B^3$ and such that there is a properly embedded disk $(D^2, \partial D^2) \hookrightarrow (B^3, \partial B^3)$ which misses $L$ and separates the two arcs. 
%A tangle $(B^3, L)$ is rational if and only if the $2$-fold cover of $B^3$ branched along $L$ is a solid torus $S^1 \times D^2$.
It is well-known that rational tangles (up to isotopy fixing the boundaries) correspond to elements of $\mathbb{Q}\cup \{ \infty \}$,
called the {\it slopes} of the rational tangles.
For example, the rational tangle of slope $\beta/\alpha$ 
can be illustrated as in Figure \ref{fig-rtangle},
where $\alpha, \beta$ are defined by the continued fraction 
\begin{eqnarray*}
\begin{array}{rrl}
\displaystyle\frac{\beta}{\alpha}&=&-a_0+[a_1, -a_2, \dots, \pm a_m]\\
&:=&-a_0+\displaystyle\frac{1}{a_1+\displaystyle\frac{1}{-a_2+\displaystyle\frac{1}{\cdots+\displaystyle\frac{1}{\pm a_m}}}}
\end{array}
\end{eqnarray*}
together with the condition that $\alpha$ and $\beta$ are relatively prime 
and $\alpha\geq 0$. 
Here, the numbers $a_i$ denote the numbers of right-hand half twists. 
%
%%%%%%%%%%%%%
\begin{figure}[btp]
\begin{center}
\includegraphics*[width=3.5cm]{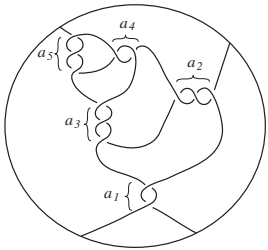}
\end{center}
\caption{$n=5, a_0=0, a_1=2, a_2=3, a_3=3, a_4=2, a_5=3$ and 
$\beta/\alpha=31/50$.}
\label{fig-rtangle}
\end{figure}
%%%%%%%%%%%%%
%

A {\it Montesinos pair} is a (3,1)-manifold pair which is built from the pair 
in Figure~\ref{fig-mont-pair}(1) or Figure~\ref{fig-mont-pair}(2)
by plugging some of the holes with rational tangles of finite slopes.
%
%%%%%%%%%%%%%
\begin{figure}[btp]
\begin{center}
\includegraphics*[width=9cm]{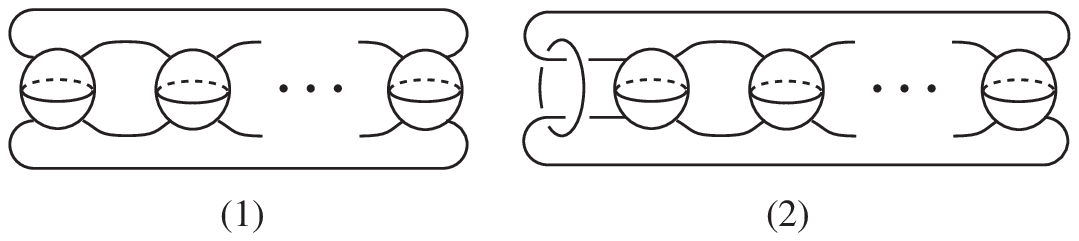}
\end{center}
\caption{}
\label{fig-mont-pair}
\end{figure}
%%%%%%%%%%%%%
%
We say that a Montesinos pair is {\it trivial} 
if it is homeomorphic to a rational tangle or $(S, P)\times I$,
where $S$ is a 2-sphere, $P$ is the union of four distinct points on $S$ 
and $I$ is a closed interval.
A {\it Montesinos link} is a link 
obtained by plugging the remaining holes of a Montesinos pair 
in Figure \ref{fig-mont-pair}(1)  
with rational tangles of finite slopes,
as shown in Figure \ref{fig-mon}. 
%
%%%%%%%%%%%%%
\begin{figure}[btp]
\begin{center}
\includegraphics*[width=5cm]{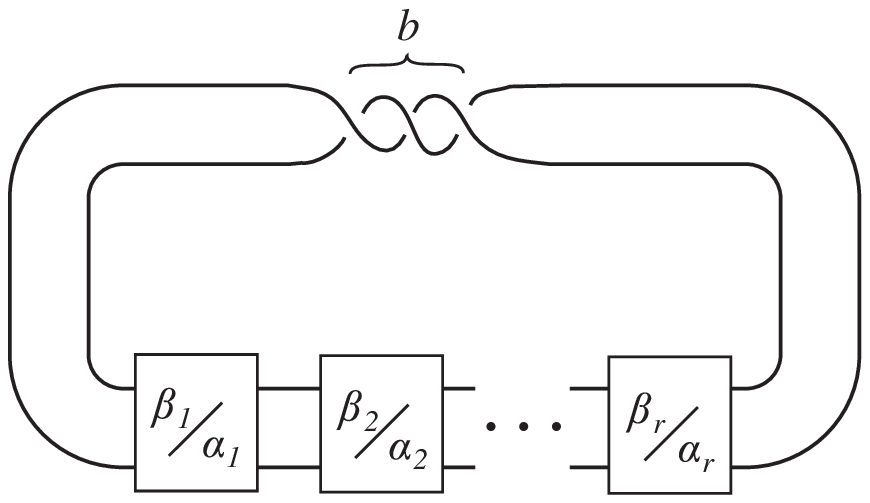}
\end{center}
\caption{$b=3$.}
\label{fig-mon}
\end{figure}
%%%%%%%%%%%%%
%
Unless otherwise stated, 
we assume that the slope $\beta_i/\alpha_i$ of each rational tangle is not an integer, that is, $\alpha_i >1$. 
The above Montesinos link is denoted by 
$L(-b; \beta_1/\alpha_1, \dots, \beta_r/\alpha_r)$. 
(We note that this is denoted by the symbol 
$m(0|b;(\alpha_1,\beta_1),(\alpha_2,\beta_2),\dots,(\alpha_r,\beta_r))$ in \cite{Boi3}.)
%A Montesinos link is said to be {\it elliptic}
%if it is a nontrivial 2-bridge link or 
%if $r=3$ and $\frac{1}{\alpha_1}+\frac{1}{\alpha_2}+\frac{1}{\alpha_3}>1$.
An {\it arborescent link} is a link in $S^3$
obtained by gluing some Montesinos pairs in their boundaries
as in Figure \ref{fig-alink3}, see \cite{Bon}.
%
%%%%%%%%%%%%%
\begin{figure}[btp]
\begin{center}
\includegraphics*[width=8cm]{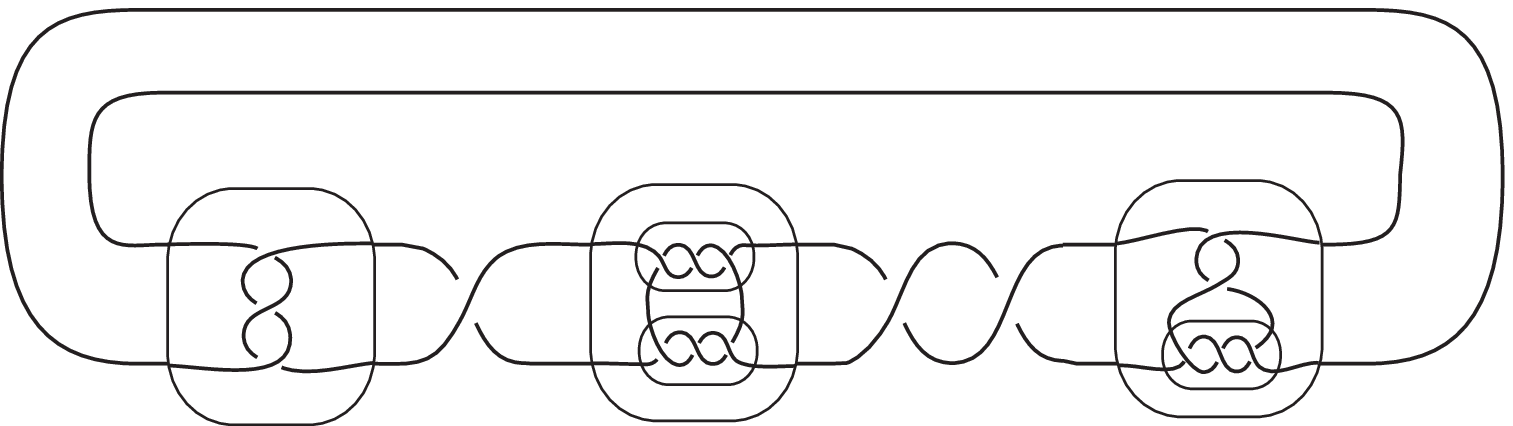}
\end{center}
\caption{}
\label{fig-alink3}
\end{figure}
%%%%%%%%%%%%%
%

%
\vspace{2mm}

The main result of this section is the following proposition which is used to prove Theorem \ref{thm-main} in Section \ref{sec-thm} when the link $L$ is an arborescent link.

\begin{proposition}\label{prop-1}
Let $L$ be an arborescent link which is not a generalized Montesinos link, 
and suppose that $w(L)=3$.
Then $L$ is hyperbolic.
\end{proposition}

\begin{proof} %[Proof of Proposition \ref{prop-1}]
Let $L$ be an arborescent link which is not a generalized Montesinos link,
and suppose that $w(L)=3$.
Assume on the contrary that $L$ is not hyperbolic.
By \cite{Bon} (cf. \cite{Fut} or \cite[Proposition 3]{Jan2}), 
$L$ is equivalent to one of the links in Figure \ref{fig-non-hyp},
%
%%%%%%%%%%%%%
\begin{figure}[btp]
\begin{center}
\includegraphics*[width=9cm]{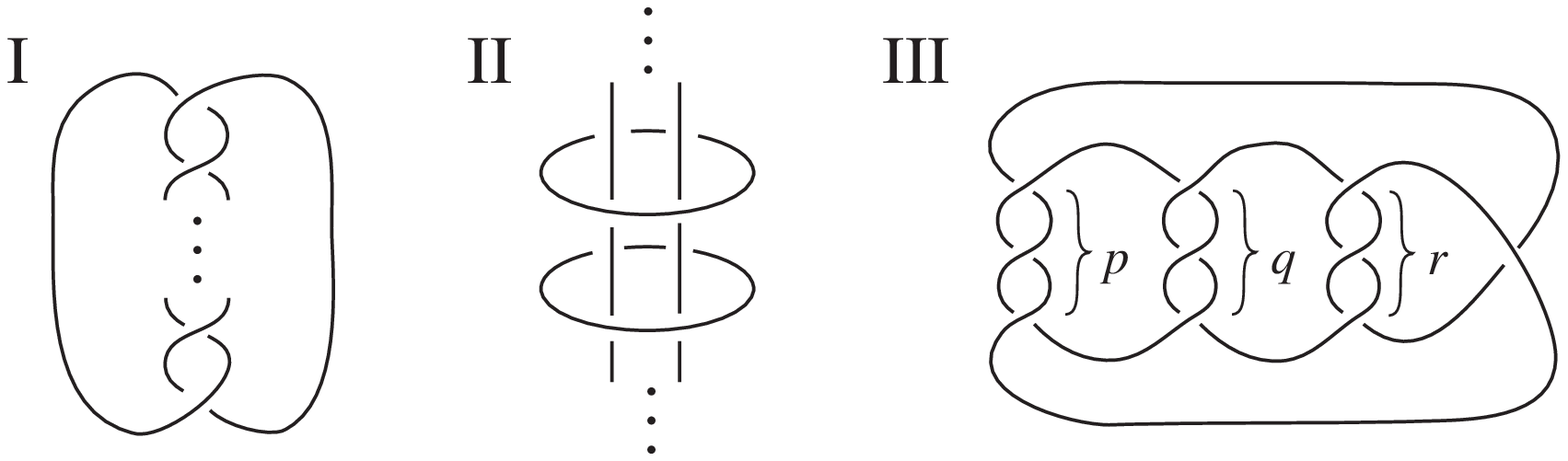}
\end{center}
\caption{}
\label{fig-non-hyp}
\end{figure}
%%%%%%%%%%%%%
%
namely, one of the following holds.
\begin{itemize}
\item[I.] $L$ is a torus knot or link of type $(2,n)$ for some integer $n$,
\item[II.] $L$ has two parallel components, each of which bounds a twice-punctured disk properly embedded in $S^3\setminus L$,
\item[III.] $L$ or its reflection is the pretzel link $P(p,q,r,-1):=L(-1;1/p,1/q,1,r)$, where $p,q,r\geq 2$ and $\displaystyle{\frac{1}{p}+\frac{1}{q}+\frac{1}{r}\geq 1}$.
\end{itemize}
By the assumptions that $L$ is not a generalized Montesinos link and that $w(L)=3$,
$L$ must be equivalent to a link in Figure \ref{fig-non-hyp} II,
namely, $L$ has two parallel components, 
each of which bounds a twice-punctured disk properly embedded in $S^3\setminus L$.
Moreover, since $w(L)=3$, $L$ must have 3 components.
Recall that the 2-fold branched cover of $S^3$ branched along $L$ is a graph manifold.
By \cite[Proposition 20 (2)]{Boi8}, the union of any two components of $L$ is a 2-bridge link.
Then, by arguments in the proof of \cite[Proposition 4 (1)]{Jan2},
we see that $L$ is equivalent to the link in Figure \ref{fig-except}.
%
%%%%%%%%%%%%%
\begin{figure}[btp]
\begin{center}
\includegraphics*[width=3cm]{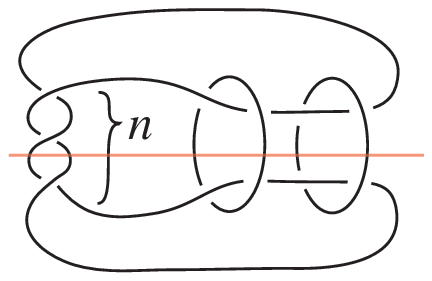}
\end{center}
\caption{}
\label{fig-except}
\end{figure}
%%%%%%%%%%%%%
%
However, the link in the figure is a generalized Montesinos link, which contradicts the assumption.
Hence, $L$ is hyperbolic.
\end{proof}

%%%%%%%%%%%%%%%%%%%%%%%%%%%%%%%%%%%%%%%%%%%%%%%%%%%%%%%%

\section{Proof of Theorem \ref{thm-main} for arborescent links}\label{sec-thm}

Let $L$ be an arborescent link and suppose that $w(L)=3$.
If $L$ is a generalized Montesinos link, then we have $b(L)=3$ by \cite{Boi3}.
Thus we assume that $L$ is not a generalized Montesinos link in the remainder of this proof.
Then, by Proposition \ref{prop-1}, $L$ is hyperbolic.
Let $P=P_1\cup\cdots\cup P_k$ be the union of Conway spheres which gives the characteristic decomposition of $L$ 
(see \cite{Bon} for definition of the characteristic decompositions of link: this decomposition correspond to the geometric decomposition of the 3-orbifold with underlying space $S^3$ and singular locus $L$ with branching index 2, see \cite{BMP}).
Let $M:=M_2(L)$ be the 2-fold cover of $S^3$ branched along $L$,
and let $T_i$ be the pre-image of $P_i$ in $M$ ($i=1,\dots,k$).
Then each $T_i$ is a separating torus in $M$ and $T=T_1\cup\cdots\cup T_k$ gives the JSJ decomposition of $M$ (cf. \cite[Proposition 4]{Jan2}).
Let $\tau_L$ be the covering involution of the 2-fold branched cover.
By construction, the following hold.
\begin{itemize}
\item[(T1)] each $T_i$ is $\tau_L$-invariant and $\tau|_{T_i}$ is hyper-elliptic, and 
\item[(T2)] $\tau_L$ preserves each JSJ piece and each exceptional fiber of Seifert pieces.
\end{itemize}

Recall that we have an exact sequence
$$
1\rightarrow \pi_1(M)\rightarrow \pi_1(S^3\setminus L)/N \rightarrow \Z/2\Z\rightarrow 1,
$$
where $N$ is the subgroup of $\pi_1(S^3\setminus L)$ normally generated by the squares of the meridians.
%(the quotient group $\pi_1(S^3\setminus L)/N$ is called the {\it $\pi$-orbifold group} of $L$). 
Let $m_1$, $m_2$ and $m_3$ be meridians of $\pi_1(S^3\setminus L)$ generating the group.
For $1\le i\le 3$ we denote the image of $m_i$ in $\pi_1(S^3\setminus L)/N$ again by $m_i$. 
Since $\pi_1(M)$ can be regarded as an index-2 subgroup of $\pi_1(S^3\setminus L)/N$ by the above exact sequence,
any element of $\pi_1(M)$ can be represented as a product of even numbers of $m_1$, $m_2$ and $m_3$'s.
Set $g_1:=m_1m_2$ and $g_2:=m_1m_3$.
Then $g_1$ and $g_2$ generate $\pi_1(M)$.
Let $\alpha$ be the automorphism of $\pi_1(S^3\setminus L)/N$ induced by the conjugation by $m_1$.
Then $\tau_L$ is a realization of $\alpha$.
We see $\alpha(g_i)=m_1g_im_1^{-1}=g_i^{-1}$ for each $i=1,2$,
and hence, $\alpha|_{\pi_1(M)}$ is an automorphism of $\pi_1(M)$ which sends each generator $g_i$ to $g_i^{-1}$.
Namely, $\alpha$ is an {\it inversion} of $\pi_1(M)$ (cf. \cite{Boi8}).
Since $M$ is a graph manifold which admits an inversion, the Heegaard genus of $M$ is $2$ by \cite[Theorem 3]{Boi8}.
Recall that $T=T_1\cup\cdots\cup T_k$ gives the nontrivial JSJ decomposition of $M$, where each $T_i$ is a separating torus in $M$.
By \cite[Proposition 4]{Jan2}, $M$ satisfies one of the following conditions (M1), (M2), (M3) and (M4) which originally come from \cite{Kob}.
\begin{itemize}
\item[(M1)] $M$ is obtained from a Seifert fibered space $M_1$ over a disk with two exceptional fibers and the exterior $M_2$ of a non-hyperbolic 1-bridge knot $K$ in a lens space by gluing their boundaries so that the meridian of $K$ is identified with the regular fiber of $M_1$.
\item[(M2)] $M$ is obtained from a Seifert fibered space $M_1$ over a disk with two or three exceptional fibers and the exterior $M_2$ of a non-hyperbolic 2-bridge knot $K$ in $S^3$ by gluing their boundaries so that the meridian of $K$ is identified with the regular fiber of $M_1$.
\item[(M3)]  $M$ is obtained from a Seifert fibered space $M_1$ over a M\"{o}bius band with one or two exceptional fibers and the exterior $M_2$ of a non-hyperbolic 2-bridge knot $K$ in $S^3$ by gluing their boundaries so that the meridian of $K$ is identified with the regular fiber of $M_1$.
\item[(M4)] $M$ is obtained from two Seifert fibered spaces $M_1$ and $M_2$ over a disk with two exceptional fibers and the exterior $M_3$ of a non-hyperbolic 2-bridge link $L=K_1\cup K_2$ in $S^3$ by gluing $\partial (M_1\cup M_2)$ and $\partial M_3$ so that the meridian of $K_i$ is identified with the regular fiber of $M_i$ ($i=1,2$).
\end{itemize}

Assume that $M$ satisfies the condition (M1).
That is, $M$ is obtained from a Seifert fibered space $M_1$ over a disk with two exceptional fibers and the exterior $M_2$ of a non-hyperbolic 1-bridge knot $K$ in a lens space by gluing their boundaries so that the meridian of $K$ is identified with the regular fiber of $M_1$.
By \cite{Kob}, $M_2$ satisfies one of the following.
\begin{itemize}
\item[(M1-a)] $M_2$ is a Seifert fibered space over a disk with two exceptional fibers, or
\item[(M1-b)] $M_2$ is a Seifert fibered space over a M\"{o}bius band with one exceptional fiber.
\end{itemize}

First we assume that $M_2$ satisfies (M1-a).
%Let $\beta_i/\alpha_i$ and $\beta_i'/\alpha_i'$ be rational numbers
%such that $M_i=D(\beta_i/\alpha_i,\beta_i'/\alpha_i')$ w.r.t. $h_i$ and $m_i$ for $i=1,2$.
Recall that the covering involution $\tau_L$ satisfies the conditions (T1) and (T2).
Since the center of $\pi_1(M)$ is trivial, the strong equivalence class of $\tau_L$ is determined by its image in the mapping class group by \cite[Theorem 7.1]{Tol}. 
By \cite[Lemma 4 (1)]{Jan2} (or \cite[Proposition 6 (1)]{Jan2}), we may assume that the restriction $\tau_L|_{M_i}$ $(i=1,2)$ is a fiber-preserving involution of $M_i$ which induces the involution on the base orbifold as illustrated in Figure \ref{fig-inv1}.
%
%%%%%%%%%%%%%
\begin{figure}[btp]
\begin{center}
\includegraphics*[width=4cm]{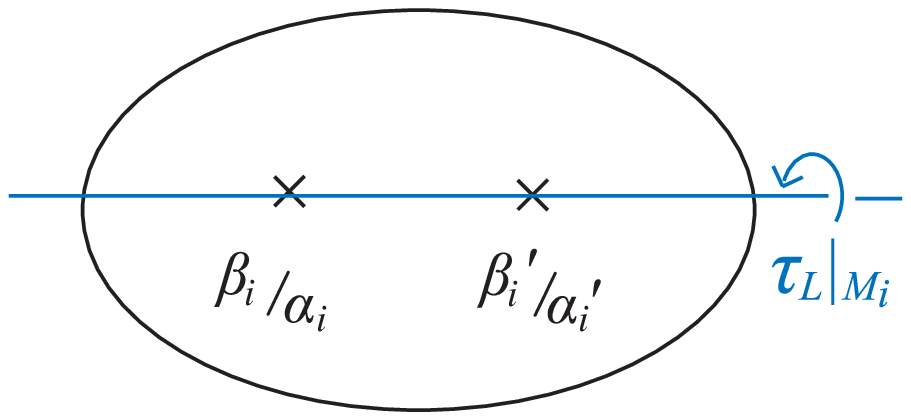}
\end{center}
\caption{}
\label{fig-inv1}
\end{figure}
%%%%%%%%%%%%%
%
Note that each quotient orbifold $(M_i,\fix \tau_L|_{M_i})/\tau_L|_{M_i}$ ($i=1,2$) is a Montesinos pair with two rational tangles.
By gluing them so that the image of the meridian of $K$ is identified with the image of the regular fiber of $M_1$, we see that $L$ must be a 3-bridge link in Figure \ref{fig-l1}
(see also \cite[Section 7, Case 1.1]{Jan2}).
%
%%%%%%%%%%%%%
\begin{figure}[btp]
\begin{center}
\includegraphics*[width=3cm]{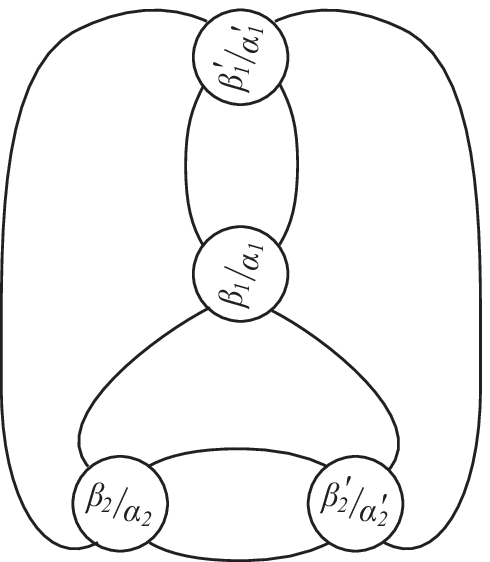}
\end{center}
\caption{}
\label{fig-l1}
\end{figure}
%%%%%%%%%%%%%
%

Assume that $M_2$ satisfies (M1-b).
By \cite[Lemma 4 (1) and (2)]{Jan2} together with \cite[Theorem 7.1]{Tol}, 
we may assume that the restriction $\tau_L|_{M_i}$ $(i=1,2)$ is a fiber-preserving involution of $M_i$ which induces the involution on the base orbifold as illustrated in Figure \ref{fig-inv2} ($i$).
%
%%%%%%%%%%%%%
\begin{figure}[btp]
\begin{center}
\includegraphics*[width=7cm]{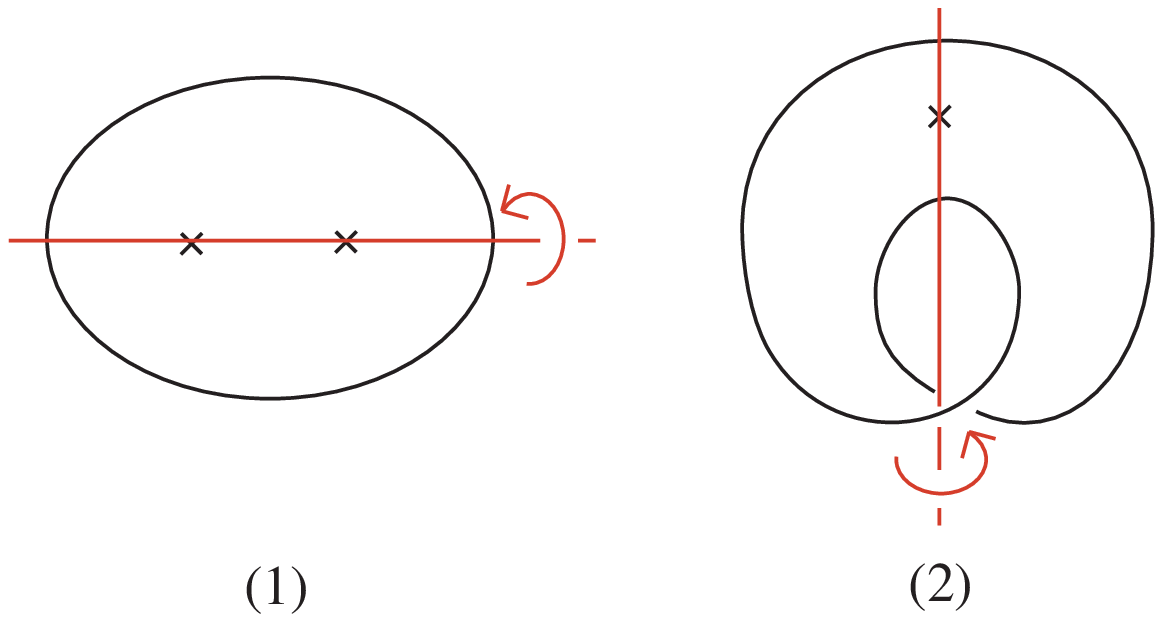}
\end{center}
\caption{}
\label{fig-inv2}
\end{figure}
%%%%%%%%%%%%%
%
By considering the quotient orbifold $(M, \fix\tau_L)/\tau_L$, we see that $L$ is equivalent to a 3-bridge link in Figure \ref{fig-l2}
(see also \cite[Section 7, Case 1.2]{Jan2}).
%
%%%%%%%%%%%%%
\begin{figure}[btp]
\begin{center}
\includegraphics*[width=3cm]{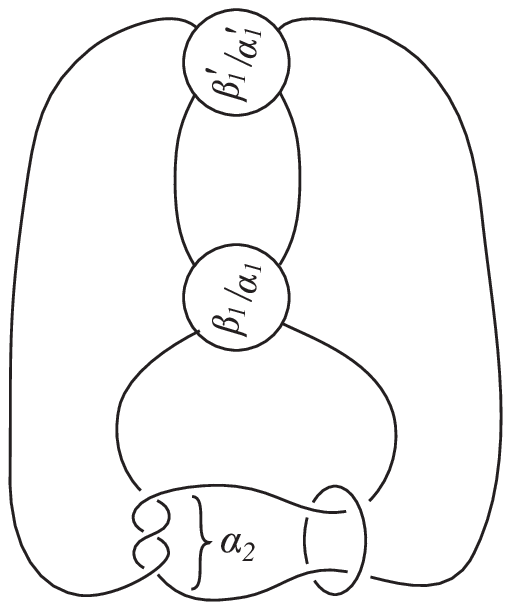}
\end{center}
\caption{}
\label{fig-l2}
\end{figure}
%%%%%%%%%%%%%
% 

The remaining cases can be treated similarly except for the case where $M$ satisfies the condition (M3).
Thus, in the rest of this section, we assume that $M$ satisfies the condition (M3).
That is, $M$ is obtained from a Seifert fibered space $M_1$ over a M\"{o}bius band with one or two exceptional fibers and the exterior $M_2$ of a non-hyperbolic 2-bridge knot $K$ in $S^3$ by gluing their boundaries so that the meridian of $K$ is identified with the regular fiber of $M_1$.
By an argument similar to those for the previous cases, we can see that $L$ is equivalent to the link in Figure \ref{fig-l3}.
%
%%%%%%%%%%%%%
\begin{figure}[btp]
\begin{center}
\includegraphics*[width=3cm]{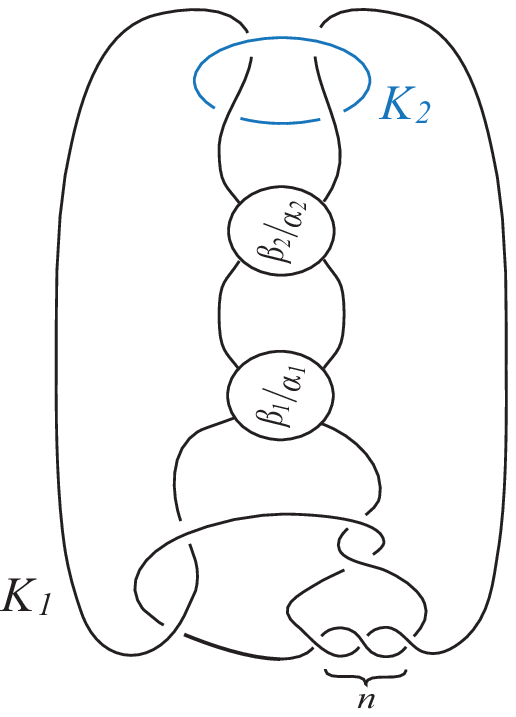}
\end{center}
\caption{}
\label{fig-l3}
\end{figure}
%%%%%%%%%%%%%
% 
For the link in Figure \ref{fig-l3}, we may assume that the rational number $\beta_1/\alpha_1$ is not an integer, and that the rational number $\beta_2/\alpha_2$ is an integer or not an integer according to whether the number of the exceptional fibers of $M_1$ is one or two.
We can see that the bridge number of the link $K_1\cup K_2$ in the figure is at least 4, since $K_1$ is a 3-bridge link by \cite{Boi3} and \cite{Jan2}.
However, by \cite[Lemma 1.7]{Boi3} and \cite[Corollary 3.3]{Boi6}, we have $w(K_1\cup K_2)\geq w(K_1)+w(K_2) = 3+1=4$, which contradicts the assumption that $w(L)=3$.

This completes the proof of Theorem \ref{thm-main} for arborescent knots.

%%%%%%%%%%%%%%%%%%%%%%%%%%%%%%%%%%%%%%%%%%%%%%%%%%%%%%%%

\section{Subgroups generated by meridians}\label{sec-free}

In this section  we study subgroups of knot and link groups that are generated by two or three meridians and we give a proof of Theorem \ref{thm-3meridians}.

\smallskip Let $L$ be a link in $S^3$ and $E(L)$ be the link space. Choose annuli and tori as follows:
\begin{enumerate}
\item Let $\{A_1,\ldots ,A_n\}$ be a maximal collection of non-parallel and properly embedded essential annuli in $E(L)$. Thus the closures of the components of ${E(L)\backslash \underset{1\le i\le n}{\cup}A_i}$ are the link spaces $E(L_1),\ldots E(L_k)$ where the $L_i$ are  the prime factors of $L$.
\item Let $\{T_1,\ldots ,T_m\}$ be the union of the characteristic families of tori of the manifolds $E(L_i)$ for $1\le i\le n$.
\end{enumerate}

Thus the closures of the components of $$E(L)\backslash \left(\left( \underset{1\le i\le n}{\cup}A_i\right)\cup \left(\underset{1\le i\le m}{\cup}T_i\right)\right)$$ are the pieces of the JSJ-decompositions of the link spaces $E(L_i)$ with $1\le i\le n$.  We call such a piece peripheral if it meets a boundary component of $E(L)$.

\medskip Let now $G=\pi_1(E(L))$. Let $\mathbb A_L$ be the graph of group decomposition of $G$ corresponding to the splitting of $E(L)$ along the $A_i$ and $T_i$. Thus the vertex groups are the fundamental groups of pieces of the JSJ-decompositions of the $E(L_i)$ and the edge groups are infinite cyclic or isomorphic to $\mathbb Z^2$.

\begin{lemma}\label{lemNE} Let $L$ be as above, $G:=\pi_1(E(L))$  and $m_1,\ldots ,m_k\in G$ be meridians (not necessarily corresponding to the same component of $L$).

Then either $\langle m_1,\ldots ,m_k\rangle$ is free or there exist meridians $m_1',\ldots,m_k'\in G$ such that the following hold:
\begin{enumerate}
\item  $(m_1,\ldots ,m_k)$ is Nielsen-equivalent to $(m_1',\ldots ,m_k')$ and $m_i$ is conjugate to $m_i'$ for $1\le i\le k$.
\item  There exist $i\neq j\in\{1,\ldots ,k\}$ such that $\langle m_i',m_j'\rangle$ is conjugate to a vertex group of $\mathbb A_L$ that corresponds to a peripheral piece of some $E(L_i)$. Moreover $m_i'$ and $m_j'$ are conjugate to meridians in this vertex group.
\end{enumerate}
\end{lemma}

\begin{proof} We consider the action of $G$ on the Bass-Serre tree $T$ corresponding to $\mathbb A_L$. Any $m_i$ acts elliptically and the fixed point set of $m_i$ coincides with the fixed point set of $m_i^n$ for any $n\neq 0$. This is true as $m_i$ is a peripheral element and therefore not a proper root of the regular fiber of any Seifert piece.

Moreover for all $i\in\{1,\ldots ,k\}$ the element $m_i$ (and therefore also $m_i^n$ with $n\neq 0$) fixes no edge corresponding to a canonical torus of the JSJ-decom\-po\-sition of some $E(L_i)$ as no power of the meridian is freely homotopic to a curve in one of these tori.% (but possible an edge corresponding to an annulus of the JSJ if $K$ in not prime.

It now follows from Theorem~7 of \cite{W} applied to $(\{m_1\},\ldots ,\{m_k\},\emptyset)$ that either $\langle m_1,\ldots ,m_k\rangle$ is free or that there exist elements $m_1',\ldots,m_k'$ such that the following hold:

\begin{enumerate}
\item $(m_1,\ldots ,m_k)$ is Nielsen-equivalent to $(m_1',\ldots ,m_k')$.
\item $m_i$ is conjugate to $m_i'$ for $1\le i\le k$.
\item There exist $i\neq j\in\{1,\ldots ,k\}$ such that nontrivial powers of $m_i'$ and $m_j'$ fix a common vertex of $T$.
\end{enumerate}
This implies in particular that $m_i'$ is a meridian for $1\le i\le k$. The above remark further implies that not only powers of $m_i'$ and $m_j'$ but $m_i'$ and $m_j'$ themselves fix a common vertex $v$ of $T$  that is therefore also fixed by $\langle m_i',m_j'\rangle$. As both $m_i'$ and $m_j'$ only fix vertices of $T$ that correspond to peripheral pieces it follows that $v$ corresponds to a peripheral piece. As no meridian is conjugate in a peripheral piece to an element corresponding to one of the characteristic tori it follows moreover that $m_i'$ and $m_j'$ are conjugate to meridians in the stabilizer of $v$.%and are in the correspnding groups conjugate to a meridian.
\end{proof}

\begin{proposition}\label{propNE} Let $K$ be a knot and  $G:=\pi_1(E(K))$. If $m_1,m_2\in G$ are meridians that generate a non-free subgroup of $\pi_1(E(K))$ then $K$ has a prime factor $K_1$ that is a 2-bridge knot and $\langle m_1,m_2\rangle$ is conjugate to the subgroup of $G$ corresponding to~$K_1$. 
\end{proposition}

\begin{proof} It follows from Lemma~\ref{lemNE} that $\langle m_1,m_2\rangle$ lies in the subgroup corresponding to a peripheral piece of $E(K)$. Thus $\langle m_1,m_2\rangle$ is contained in the subgroup corresponding to the peripheral piece $M$ of the JSJ-decomposition of a prime factor $K_1$ of $K$. Moreover $m_1$ and $m_2$ are in this subgroup conjugate to the meridian. We distinguish two cases:

\smallskip Suppose first that $M$ is Seifert fibered. Thus $M$ is a torus knot space or a cable space. In the first case it follows from~\cite{Ros} that either $\langle m_1,m_2\rangle$ is free or that $\langle m_1,m_2\rangle=\pi_1(M)$ and that $M$ is the exterior of a 2-bridge link which proves the claim. In the second case $M$ is the mapping torus of a disk with finitely many punctures with respect to an automorphism of finite order. Moreover (like all elements conjugate to a meridian) both $m_1$ and $m_2$ lie in the free fundamental group of the fiber which implies that $\langle m_1,m_2\rangle$ is free.

\smallskip Suppose now that $M$ is hyperbolic. We may assume that $\langle m_1,m_2\rangle$ is not Abelian as two conjugates of the meridian that generate an Abelian group must lie in the same conjugate of the same peripheral subgroup and therefore generate a cyclic subgroup.

It follows from Proposition~2 of \cite{Boi9} that either $\langle m_1,m_2\rangle=\pi_1(M)$ and that $M$ is the exterior of a 2-bridge knot or that $|\pi_1(M):\langle m_1,m_2\rangle|=2$ and the 2-sheeted cover $\tilde M$ of $M$ corresponding to $\langle m_1,m_2\rangle$ is the exterior of a 2-bridge link with 2 components.

In the first case the conclusion is immediate. Suppose now that the second case occurs. As $m_1$ and $m_2$ are conjugate in $\pi_1(M)$ it follows that both boundary components of $\tilde M$ cover the same boundary component of $M$, in particular $M$ is a knot exterior. Now  $\langle m_1,m_2\rangle$  contains a conjugate of the peripheral subgroup of $\pi_1(M)$ and is normal in $\pi_1(M)$. It follows that $\langle m_1,m_2\rangle$ contains all parabolic elements of $\pi_1(M)$. As $\pi_1(M)$ is a knot group, it is  generated by parabolic elements. It follows that $\pi_1(M)=\langle m_1,m_2\rangle$ which yields a contradiction.
\end{proof}

The remainder of this section is devoted to the proofs of Theorem \ref{thm-3meridians} and Corollary \ref{cor-3meridians}.

\begin{proof}[Proof of Theorem \ref{thm-3meridians}] It follows from Lemma~\ref{lemNE} that we may assume that $\langle m_1,m_2\rangle$  fixes a vertex $v$ of the Bass-Serre tree that corresponds to the peripheral piece $M$ of $S^3\backslash K$. By Proposition~\ref{propNE} the group $\langle m_1,m_2\rangle$ is free. 

Choose a torus $T$ of the characteristic family of tori for $S^3\backslash K$ such that $T$ cuts $S^3\backslash K$ into two pieces, a geometric knot space $N$ and its complement $\overline M$. Clearly $M$ is contained in $\overline M$. Note that for homology reasons the subgroup $\langle g_1,g_2\rangle$ intersects any conjugate of the free Abelian subgroup $A$ of $G=\pi_1(S^3\backslash K)$ corresponding to $T$ at most in a cyclic subgroup that is a subgroup of the cyclic group generated by the meridian of $N$. Consider the action of $G$ on the Bass-Serre tree corresponding to the amalgamated product $\pi_1(N)*_A\pi_1(\overline M)$. Let $v$ be the vertex fixed by $\langle m_1,m_2\rangle$, note that $v$ corresponds to $\pi_1(\overline M)$.

As the meridian of $N$ does not agree with the fiber of $N$ (if $N$ is Seifert fibered) it follows that no element of $\langle m_1,m_2\rangle$ fixes a vertex in distance more than $1$ from $v$. Moreover $m_3$ fixes a single vertex that corresponds to $\pi_1(\overline M)$. Applying Theorem~7 of \cite{W} to $(\{m_1,m_2\},\{m_3\})$ follows that either $m_3$ also fixes $v$ or that $\langle m_1,m_2,m_3\rangle \cong \langle m_1,m_2\rangle*\langle m_3\rangle\cong F_3$. This proves the claim.
\end{proof}

Corollary \ref{cor-satellite} is a direct consequence of Theorem \ref{thm-3meridians}.  We prove now Corollary~\ref{cor-3meridians}.

\begin{proof}[Proof of Corollary \ref{cor-3meridians}]
Let $K\subset S^3$ be a knot such that $w(K) = 3$. If $K$ is prime, then Theorem \ref{thm-3meridians} implies that $K$ is a hyperbolic knot or a torus knot. If $K= K_1 \sharp K_2$ is a nontrivial connected sum,  then the 2-fold cover $M_2(K)$ of $S^3$ branched along $K$ is the nontrivial connected sum $M_2(K_1) \sharp M_2(K_2)$ of the 2-fold branched covers of $K_1$ and $K_2$. Since $w(K) = 3$, it follows that $\pi_{1}(M_2(K))$ is generated by two elements. Since $\pi_{1}(M_2(K)) = \pi_{1}(M_2(K_1)) * \pi_{1}(M_2(K_2))$ is a free product of nontrivial groups, by the orbifold theorem, see \cite{BP}, it follows that each group $\pi_{1}(M_2(K_1))$ and $\pi_{1}(M_2(K_2))$ is cyclic. Again the orbifold theorem allows to conclude that $K_1$ and $K_2$ are 2-bridge knots.
\end{proof}

%%%%%%%%%%%%%%%%%%%%%%%%%%%%%%%%%%%%%%%%%%%%%%%%%%%%%%%%

\section{Proof of Theorem \ref{thm-main}}\label{sec-final}

Let $L$ be a link in $S^3$, and suppose that the 2-fold branched cover $M:=M_2(L)$ of $S^3$ branched along $L$ is a graph manifold.
Since we have already treated the case when $L$ is an arborescent link in Section \ref{sec-thm}, we  assume here that $L$ is not an arborescent link
and that $w(L)=3$.

We first assume that $M$ is a Seifert fibered space.
Then $L$ is either a (generalized) Montesinos link or a Seifert link, i.e., $S^3\setminus L$ admits a Seifert fibration.
If $L$ is a (generalized) Montesinos link or a torus link, then we have $b(L)=3$ by \cite{Boi3, Ros}.
So we assume that $L$ is a Seifert link which is not a torus link.
By \cite{Bur2}, we see that $L$ is the union of a torus knot of type $(2,b)$ and its {\it core} of index 2, in which case it is easy to see that $b(L)=3$.

Next we assume that $M$ is not a Seifert fibered space. 
Let $T=T_1\cup\cdots\cup T_k$ be tori which give the JSJ decomposition of $M$.
As in Section \ref{sec-thm}, we can see that $M$ is a genus-2 manifold and the covering involution $\tau_L$ is a realization of an inversion of $\pi_1(M)$.
Let $\alpha:=(\tau_L)_{\ast}$ be the automorphism of $\pi_1(M)$ and let $g$ and $h$ be a pair of generators for $\pi_1(M)$.
By \cite[Proposition 20]{Boi8}, $\tau_L$ respects the JSJ decomposition of $M$ and the Seifert fibered structures on the JSJ pieces.
Let $Q$ be the oriented circle bundle over the M\"{o}bius band.
We follow the argument in \cite[Section 3]{Boi9}, under the assumption that $M$ is a genus-2 closed manifold.
We first deal with the following case.

\subsection{The JSJ decomposition has a separating torus and no piece homeomorphic to $Q$}

Let $T_1$ be the separating torus by changing order if necessary, and let $M_A$ and $M_B$ be the two submanifold of $M$ divided by $T_1$.
By the argument in \cite{Boi9}, we see that $M_A$ is a Seifert fibered space, $g$ is a root of a fiber of $M_A$ and $g^n\in\pi_1(T_1)$. 
Moreover, one of the following holds.
\begin{itemize}
\item[(i)] $M_A$ is a Seifert fibered space over a disk with two exceptional fibers and $M_B$ is the exterior of a 1-bridge knot in a lens space,
\item[(ii)] $M_A$ is a Seifert fibered space over a disk with two exceptional fibers and $M_B$ is the exterior of a non-hyperbolic 2-bridge knot in $S^3$,
\item[(iii)] $M_A$ is a Seifert fibered space over a disk with two exceptional fibers and $M_B$ is decomposed by $T_2$ into two pieces $M_B^{(1)}$ and $M_B^{(2)}$, where $M_B^{(1)}$ is the exterior of a 2-component non-hyperbolic 2-bridge link in $S^3$ and $M_B^{(2)}$ is a Seifert fibered space over a disk with two exceptional fibers,
\item[(iv)] $M_A$ is a Seifert fibered space over a M\"{o}bius band with one or two exceptional fibers and $M_B$ is the exterior of a non-hyperbolic 2-bridge knot in $S^3$,
\item[(v)] $M_A$ is a Seifert fibered space over a disk with three exceptional fibers and $M_B$ is the exterior of a non-hyperbolic 2-bridge knot in $S^3$.
\end{itemize}
Here, the boundaries of $M_A$ and $M_B$ are glued so that the fiber of $M_A$ is identified with the meridian of $M_B$.

First assume that (i) is satisfied. 
Since $\alpha(g^n)=g^{-n}$, we see that $\tau_L|_{T_1}$ is hyper-elliptic. 
Note that $\tau_L|_{T_1}$ extends to $M_B$ in a unique way and the quotient of $M_B$ by $\tau_L|_{M_B}$ gives a tangle in Figure \ref{fig-case1-i} (2) (see \cite[Lemma 9]{Jan2}).
%
%%%%%%%%%%%%%
\begin{figure}[btp]
\begin{center}
\includegraphics*[width=7cm]{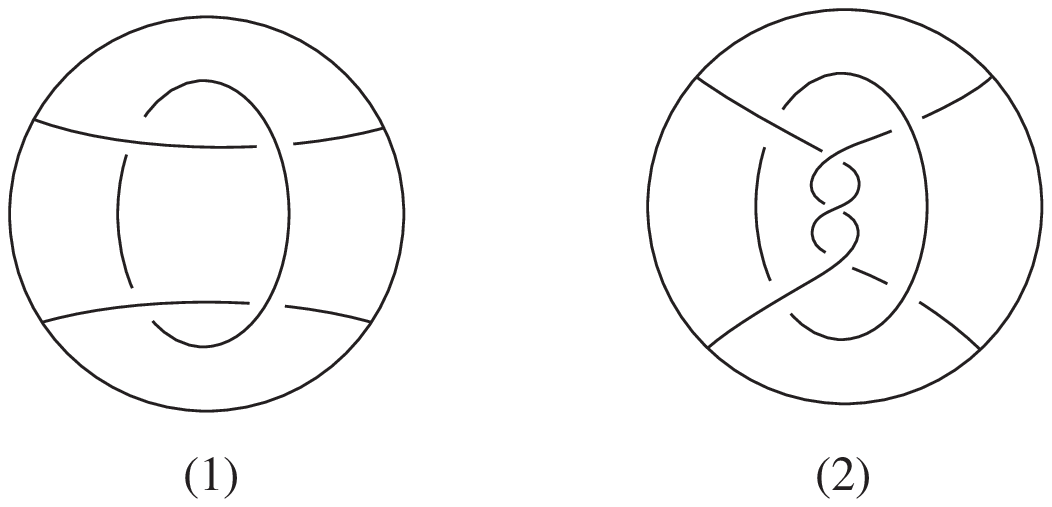}
\end{center}
\caption{}
\label{fig-case1-i}
\end{figure}
%%%%%%%%%%%%%
%
Since we assume that $L$ is not an arborescent link, we see that $\tau_L$ exchanges the two exceptional fibers of $M_A$.
(This implies that the two exceptional fibers of $M_A$ have the same index.)
Then the quotient of $M_A$ by $\tau_L|_{M_A}$ is obtained from the tangle in Figure \ref{fig-case1-i} (1) by applying Dehn surgery along the loop component in the tangle, where the surgery slope is the reciprocal of the index of the exceptional fibers of $M_A$.
Hence the quotient of $M$ by $\tau_L$ is a nontrivial lens space, a contradiction.

Assume that (ii) is satisfied. 
Note that $M_B$ is a Seifert fibered space over a disk with two exceptional fibers of indices $1/2$ and $-n/(2n+1)$.
Thus the involution on $M_B$ which is hyper-elliptic on the boundary is unique (see \cite[Lemma 4 (1)]{Jan2} for example).
By an argument similar to that for the previous case, we can lead to a contradiction.

Assume that (iii) is satisfied. 
Then we see that either $\tau_L(T_i)=T_i$ and $\tau_L|_{T_i}$ is hyper-elliptic $(i=1,2)$ or $\tau_L(T_1)=T_2$.
In the former case, we can use arguments similar to those in the previous cases to lead to a contradiction.
In the latter case, $M_A$ and $M_B^{(2)}$ are homeomorphic and $\tau_L$ interchanges the two pieces.
Denote by $N$ the quotient of $M_B^{(1)}$ by $\tau_L|_{M_B^{(1)}}$, which is a solid torus, and denote by $F$ the image of the fixed point set.
Then the exterior of $F$ in $N$ is homeomorphic to the exterior of a torus link of type $(2,2m)$.
The quotient of $M$ by $\tau_L$, which is supposed to be $S^3$, is obtained by gluing $M_A$ and a solid torus, which implies that $M_A$ is homeomorphic to the exterior of a torus knot (see \cite{Bur2}).
Thus $L$ is a nontrivial cable knot of a torus knot.
By Corollary \ref{cor-satellite}, we have $w(L)\geq 4$, a contradiction.

Assume that (iv) is satisfied.
By arguments similar to those for the previous cases, we can see that $\tau_L|_{M_A}$ and $\tau_L|_{M_B}$ are equivalent to the involutions illustrated in Figure \ref{fig-case1-iv}. 
%
%%%%%%%%%%%%%
\begin{figure}[btp]
\begin{center}
\includegraphics*[width=8cm]{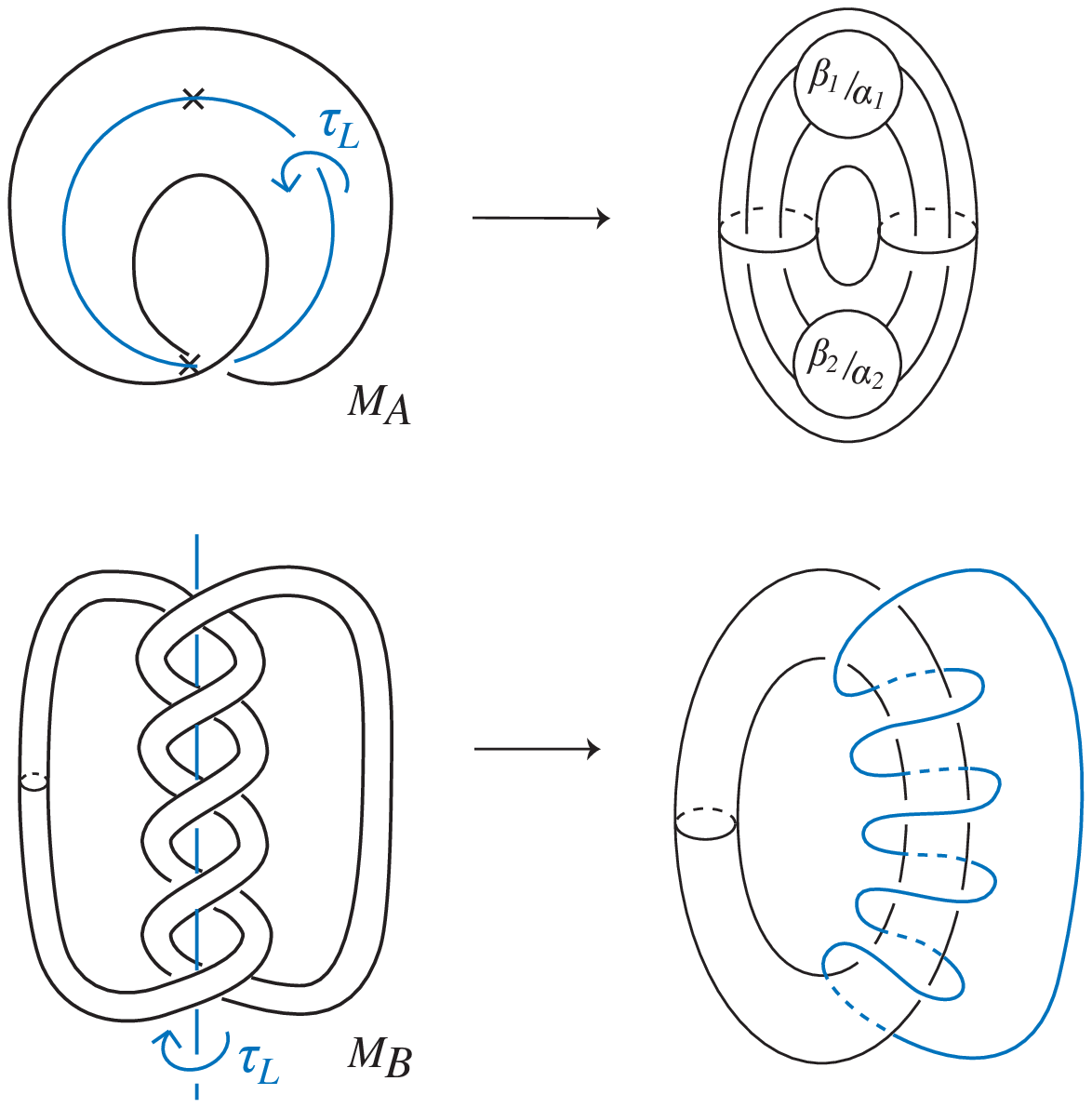}
\end{center}
\caption{}
\label{fig-case1-iv}
\end{figure}
%%%%%%%%%%%%%
%
Hence, the quotient of $M_A$ gives a 2-bridge link in a solid torus and the quotient of $M_B$ gives a component of a torus link of type $(2,2m)$ with the regular neighborhood of the other component removed.
Then we obtain the link in Figure \ref{fig-l4} (cf. \cite{Jan3}), which is a 3-bridge link.
%
%%%%%%%%%%%%%
\begin{figure}[btp]
\begin{center}
\includegraphics*[width=3cm]{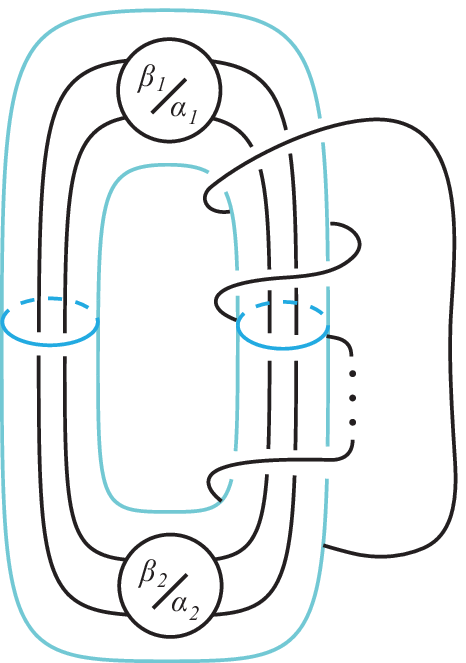}
\end{center}
\caption{}
\label{fig-l4}
\end{figure}
%%%%%%%%%%%%%
%

Assume that (v) is satisfied.
We can lead to a contradiction by arguments similar to those for the previous cases.

\subsection{The JSJ decomposition has a non-separating torus}

Since the genus of $M$ is 2, $M$ consists of one or two Seifert pieces.

We first deal with the case when $M$ consists of one Seifert piece.
By an argument of \cite{Boi9}, we have the following two cases.
\begin{itemize}
\item[(i)] The torus $T$ cuts $M$ into the exterior of a 2-component non-hyperbolic 2-bridge link, and $g$ and $hgh^{-1}$ are the meridians,
\item[(ii)] The torus $T$ cuts $M$ into a Seifert fibered space over an annulus with two exceptional fibers, whose boundary components are glued so that the fibers are identified.
\end{itemize}
When (ii) holds, $M$ is a Seifert fibered space, a contradiction.
Hence assume that (i) holds.
Note that the closure of $M\setminus T$ is a Seifert fibered space, say $M'$, over an annulus with one exceptional fiber.
Since we assume that $M$ is not a Seifert fibered space, the fibers on the two boundary components of $M'$ do not identified.
Since $g$ is a meridian of the 2-bridge link, we can see that $\tau_L|_T$ is hyper-elliptic.
Then the quotient of $M'$ by $\tau_L|_{M'}$ gives a (3,1)-manifold pair in Figure \ref{fig-case2-i}.
%
%%%%%%%%%%%%%
\begin{figure}[btp]
\begin{center}
\includegraphics*[width=4cm]{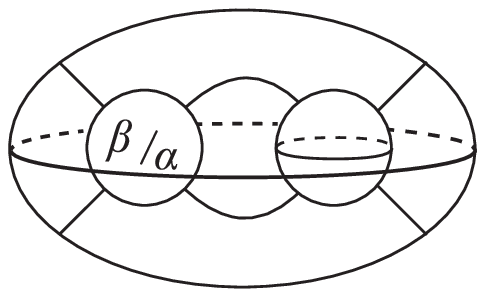}
\end{center}
\caption{}
\label{fig-case2-i}
\end{figure}
%%%%%%%%%%%%%
%
The quotient of $M$ by $\tau_L$ is obtained from $S^3\setminus (B_1\cup B_2)$, where $B_1$ and $B_2$ are open 3-balls, by gluing the two 2-spheres $\partial B_1$ and $\partial B_2$, and hence the quotient of $M$ cannot be homeomorphic to $S^3$, a contradiction.

Next we deal with the case when $M$ consists of two Seifert pieces $M_A$ and $M_B$.
By \cite{Kob}, $M_A$ is a Seifert fibered space over an annulus with one or two exceptional fibers and $M_B$ is the exterior of a 2-component non-hyperbolic 2-bridge link.
By arguments similar to those for previous cases, we can see that $L$ is equivalent to a link in Figure \ref{fig-l4} (cf. \cite{Jan3}), which is a 3-bridge link.

\subsection{There exists a piece homeomorphic to $Q$}

By \cite{Kob}, we have the following cases.
\begin{itemize}
\item[(i)] $M$ consists of two JSJ pieces homeomorphic to $Q$,
\item[(ii)] $M$ consists of two JSJ pieces one of which is homeomorphic to $Q$, and the other is either a Seifert fibered space over a disk with two exceptional fibers or a Seifert fibered space over a M\"{o}bius band with one exceptional fiber,
\item[(iii)] $M$ consists of three JSJ pieces one of which is homeomorphic to $Q$, the second piece is the exterior of a 2-component non-hyperbolic 2-bridge link and the third piece is a Seifert fibered space over a disk with two exceptional fibers.
\end{itemize}

Assume that (i) is satisfied.
By \cite[Lemma 17]{Boi9}, the regular fibers of the two pieces, considered as a Seifert fibered space over a disk with two exceptional fibers, intersect in one point, and $g^2$ is a fiber of one piece.
Then we see that $\tau_L|_T$ is hyper-elliptic, and we can lead to a contradiction by using arguments similar to those in the previous cases.

Assume that (ii) is satisfied.
By an argument in \cite[Proof of Lemma 18]{Boi9}, we can see that $\tau_L|_T$ is hyper-elliptic, and we can lead to a contradiction by using arguments similar to those in the previous cases.

Assume that (iii) is satisfied.
Similarly, we can see that either $\tau_L(T_i)=T_i$ and $\tau_L|_{T_i}$ is hyper-elliptic $(i=1,2)$ or $\tau_L(T_1)=T_2$.
In the former case, we can lead to a contradiction by using arguments similar to those in the previous cases.
In the latter case, we can see that the quotient of $M$ by $\tau_L$ is the union of $Q$ and a solid torus, which cannot be homeomorphic to $S^3$, a contradiction.

This completes the proof of Theorem \ref{thm-main}.

\begin{proof}[Proof of Corollary \ref{cor-inversion}]
Let $M$ be a closed orientable graph manifold which admits an inversion, i.e., $\pi_1(M)$ is generated by two elements $g$ and $h$ and there exists an automorphism $\alpha$ of $\pi_1(M)$ which sends $g$ and $h$ to $g^{-1}$ and $h^{-1}$, respectively.
If $M$ is a Seifert fibered space, then $\alpha$ is hyper-elliptic by \cite[Theorem 5]{Boi8}.
If $M$ is not a Seifert fibered space, then $\alpha$ is hyper-elliptic by Theorem \ref{thm-main} and \cite[Proposition 20 (3)]{Boi8}
\end{proof}

\section{Degree-one maps}\label{sec-degreeone}

In this section we prove  Proposition \ref{prop-degreeone}
 
 \begin{proof}[Proof of Proposition \ref{prop-degreeone}]\ Let $L' \subset S^3$ such that $b(L') = 3$, then $w(L') = 3$ by \cite{Boi6}. Therefore if $L \geq L'$, then $b(L) \geq w(L) \geq w(L') = 3$.

\noindent b) Let $L' \subset S^3$ such that $b(L') = 4$. Assume that $L \geq L'$ and that the 2-fold branched cover $M$ of $L$ is a graph manifold. The degree one map 
$f: E(L) \to E(L')$ between the exteriors of $L$ and $L'$ which preserves the meridians lifts to a degree one map $\tilde{f} : \tilde{E}(L) \to \tilde{E}(L')$ between their 2-fold covers,
which extends to a degree one map $\tilde{f}: M \to M'$ between their 2-fold branched covers $M:=M_2(L)$ and $M'=M_2(L')$. Since $M$ is a graph manifold, its simplicial volume $\Vert M \Vert = 0$. The existence of the 
degree one map $\tilde{f}: M \to M'$ implies that $\Vert M' \Vert  \leq \Vert M \Vert$ and thus $\Vert M' \Vert = 0$. By the orbifold theorem \cite{BP} $M'$ admits a geometric decomposition 
and thus is a connected sum of graph manifolds. Therefore $L'$ is a connected sum of links whose 2-fold branched covers are graph manifolds.

If $L'$ is prime, it follows from Corollary \ref{cor-4bridges} that $w(L') = 4$ and thus $b(L) \geq w(L) \geq w(L') = 4$.

 If $L'$ is not prime, then $L' = L'_1 \sharp L'_2$ with $b(L'_1) = 2 = w(L'_1)$ and $b(L'_2) = 3= w(L'_2)$ by \cite{Boi6}. The exterior $E(L')$ is obtained by gluing a copy of $E(L'_1)$ and of 
 $E(L'_2)$ along two boundary components of $S^1 \times P$, where $P$ is a pant. Thus one can define two epimorphisms $\phi_1: \pi_1 (E(L')) \to \pi_1 (E(L'_1))$ and 
 $\phi_2: \pi_1 (E(L')) \to \pi_1 (E(L'_2))$ such that the restriction of $\phi_1$ to $\pi_1 (E(L'_1))$ and the restriction of $\phi_2$ to $\pi_1 (E(L'_2))$ are the identity and such that $\phi_1(\pi_1 (E(L'_1))) = \mathbb{Z}$ and 
 $\phi_2(\pi_1 (E(L'_2))) = \mathbb{Z}$. These epimorphisms imply that $w(L') = w(L'_1) + w(L'_2) -1 = 4$, and thus $b(L) \geq w(L) \geq w(L') = 4$.
 \end{proof}

%==============================================================

\end{document}